\documentclass[12pt]{amsart}

\usepackage[inner=1.0in,outer=1.0in,bottom=1.0in, top=1.0in]{geometry}

\usepackage{setspace}     
\usepackage{enumerate} 	
\usepackage{afterpage} 	
\usepackage[usenames,dvipsnames]{color}
\usepackage{graphicx}	

\usepackage{amsmath} 	
\usepackage{amssymb}	
\usepackage{mathrsfs} 	
\usepackage{amsfonts} 	
\usepackage{latexsym} 	
\usepackage[latin1]{inputenc} 

\usepackage{amsthm} 	
\usepackage{stackrel}       
\usepackage{amscd} 	

\usepackage{tabularx}	
\usepackage{longtable}	

\usepackage{tabu}

\usepackage{verbatim} 
\usepackage{blindtext} 

\usepackage{amssymb,latexsym,amsmath}    
\usepackage{pstricks,multido,pst-plot}		 
\usepackage{multicol,fancyheadings}			   
\usepackage{float}                       
\usepackage[vcentermath,enableskew]{youngtab}

\usepackage{subfig}

\allowdisplaybreaks

\def\e{\epsilon}

\def\l{\lambda}

\newcommand{\Z}{\mathbb{Z}}

\def\mod{\mathop{\rm{mod}}}

\def\exp{{\mathrm{exp}}}

\def\spt{{\mathrm{spt}}}
\def\l{{\lambda}}

\numberwithin{equation}{section}

\newtheorem{theorem}{Theorem}[section]  
\newtheorem{proposition}[theorem]{Proposition}
\newtheorem{lemma}[theorem]{Lemma}
\newtheorem{corollary}[theorem]{Corollary}
\newtheorem*{conjecture}{Conjecture}

\newtheorem*{thm}{Theorem}	

\theoremstyle{definition}

\theoremstyle{remark}

\newcommand{\shortmod}{\ensuremath{\negthickspace \negthickspace \negthickspace \pmod}}

\newtheorem{ind}[]{{\rm\it Indice}}

\usepackage{multicol}

\usepackage{tikz}

\usepackage{paralist}
\usepackage{cite}

\title{Effective bounds for the Andrews spt-function}

\author[Locus Dawsey]{Madeline Locus Dawsey*}
\address{Department of Mathematics and Computer Science,
Emory University, Atlanta, GA 30322}
\email{madeline.locus@emory.edu}

\author[Masri]{Riad Masri}
\address{Department of Mathematics, Mailstop 3368, Texas A\&M University, College Station, TX 77843-3368}
\email{masri@math.tamu.edu}

\begin{document}

\thanks{*This author was previously known as Madeline Locus.}

\begin{abstract}
In this paper, we establish an asymptotic formula with an effective bound on the error term for 
the Andrews smallest parts function $\spt(n)$. We use this formula to prove recent conjectures 
of Chen concerning inequalities which involve the partition function $p(n)$ and $\spt(n)$. 
Further, we strengthen one of the conjectures, and prove that for every $\e>0$ there is an effectively computable constant $N(\e) > 0$ 
such that for all $n\geq N(\e)$, we have
\begin{equation*}
\frac{\sqrt{6}}{\pi}\sqrt{n}\,p(n)<\spt(n)<\left(\frac{\sqrt{6}}{\pi}+\e\right) \sqrt{n}\,p(n).
\end{equation*}
Due to the conditional convergence of the Rademacher-type formula for $\spt(n)$, we must employ methods which are completely 
different from those used by Lehmer to give effective error bounds for $p(n)$. Instead, our approach relies on 
the fact that $p(n)$ and $\spt(n)$ can be expressed as traces of singular moduli.
\end{abstract}

\maketitle


\section{Introduction and Statement of Results}\label{s1}

The smallest parts function $\spt(n)$ of Andrews is defined for any integer $n\geq1$ as the number of 
smallest parts among the integer partitions of size $n$.  For example, the partitions of $n=4$ are (with the smallest parts underlined):
\begin{align*}
&\underline{4},\\
&3+\underline{1},\\
&\underline{2}+\underline{2},\\
&2+\underline{1}+\underline{1},\\
&\underline{1}+\underline{1}+\underline{1}+\underline{1},
\end{align*}
and so $\spt(4)=10$.  The spt-function has many remarkable properties. For example, 
Andrews \cite{Andrews} proved the following analogues of the well-known Ramanujan congruences for the partition function $p(n)$:
\begin{align*}
\spt(5n+4)&\equiv0\pmod{5},\\
\spt(7n+5)&\equiv0\pmod{7},\\
\spt(13n+6)&\equiv0\pmod{13}.
\end{align*}

One can compute $\spt(n)$ by making use of the generating function
\begin{equation*}
\sum_{n=1}^{\infty}\spt(n)q^n=\sum_{n=1}^\infty\frac{q^n}{\left(1-q^n\right)^2\left(q^{n+1};q\right)_\infty},
\end{equation*}
where 
\begin{align*}
(a;q)_\infty=\prod_{n=0}^\infty\left(1-aq^n\right).  
\end{align*}
We use this generating function to compute the values of $\spt(n)$ required for this paper.

In this paper, we will prove the following asymptotic formula for $\spt(n)$ with an effective bound on the error term.

\begin{theorem}\label{riad} Let $\l(n):=\pi \sqrt{24n-1}/6$. Then for all $n \geq 1$, we have
\begin{equation*}
\spt(n)=\frac{\sqrt{3}}{\pi\sqrt{24n-1}}e^{\l(n)} + E_s(n), 
\end{equation*}
where
\begin{equation*}
\left|E_s(n)\right| <  (3.59 \times 10^{22})2^{q(n)}(24n-1)^2e^{\l(n)/2}
\end{equation*}
with 
\begin{align*}
q(n):=\frac{\log(24n-1)}{|\log(\log(24n-1))-1.1714|}.
\end{align*}
\end{theorem}

\vspace{0.10in}

Our interest in proving the effective error bounds of Theorem \ref{riad} was motivated in part by 
the following recent conjectures of Chen \cite{C} concerning inequalities which involve $p(n)$ and $\spt(n)$.

\begin{conjecture}[Chen]\label{Chen}
\item
\begin{enumerate}
\item\label{1} For $n\geq5$, we have $$\frac{\sqrt{6}}{\pi}\sqrt{n}\,p(n)<\spt(n)<\sqrt{n}\,p(n).$$
\item\label{2} For $(a,b)\neq(2,2)$ or $(3,3)$, we have $$\spt(a)\,\spt(b)>\spt(a+b).$$
\item\label{3} For $n\geq36$, we have $$\spt(n)^2>\spt(n-1)\,\spt(n+1).$$
\item\label{4} For $n>m>1$, we have $$\spt(n)^2>\spt(n-m)\,\spt(n+m).$$
\item\label{5} For $n\geq13$, we have $$\frac{\spt(n-1)}{\spt(n)}\left(1+\frac{1}{n}\right)>\frac{\spt(n)}{\spt(n+1)}.$$
\item\label{6} For $n\geq73$, we have $$\frac{\spt(n-1)}{\spt(n)}\left(1+\frac{\pi}{\sqrt{24}n^{3/2}}\right)>\frac{\spt(n)}{\spt(n+1)}.$$
\end{enumerate}
\end{conjecture}
\vspace{0.10in}

\textbf{Remark}.
Conjectures (\ref{1}) and (\ref{2}) are slight modifications of Chen's original claims.  

\vspace{0.10in}

By combining Theorem \ref{riad} with classical work of Lehmer \cite{L} which gives effective error bounds for $p(n)$, we will prove the following result. 

\begin{theorem}\label{main}
All of Chen's conjectures are true.
\end{theorem}

We will also use Theorem \ref{riad} to 
prove the following more precise version of Theorem \ref{main} regarding Conjecture (\ref{1}).  

\begin{theorem}[Refined Theorem \ref{main} (\ref{1})]\label{better}
For each $\e>0$, there is an effectively computable constant $N(\e) > 0$ 
such that for all $n\geq N(\e)$, we have 
\begin{align*}
\frac{\sqrt{6}}{\pi} \sqrt{n} p(n)< \spt(n) <\left(\frac{\sqrt{6}}{\pi}+\epsilon \right) \sqrt{n} p(n).
\end{align*}
\end{theorem}

\vspace{0.10in}

\textbf{Remark}.
The constant $N(\epsilon)$ of Theorem \ref{better} can be computed in practice. For example, 
by letting $\epsilon = 1 - \sqrt{6}/\pi$ in Theorem \ref{better}, we get 
Theorem \ref{main} (1) for $n \geq N(1-\sqrt{6}/\pi)$ with $N(1-\sqrt{6}/\pi)=5310$. We then 
use a computer to verify Theorem \ref{main} (1) in the exceptional range 
$5 \leq n < 5310$. 

\vspace{0.10in}

\textbf{Remark}. In analogy with the Hardy-Ramanujan asymptotic for $p(n)$, Bringmann \cite{B} used the circle method to 
establish the asymptotic 
\begin{equation*}
\spt(n) \sim \frac{1}{\pi\sqrt{8n}} e^{\pi\sqrt{\frac{2n}{3}}}
\end{equation*}
as $n\rightarrow\infty$. Bringmann's asymptotic for $\spt(n)$ implies that 
\begin{align}\label{sptpn}
\spt(n) \sim \frac{\sqrt{6}}{\pi} \sqrt{n} p(n)
\end{align}
as $n \rightarrow \infty$. Theorem \ref{better} refines the asymptotic (\ref{sptpn}).  
\vspace{0.10in}

We now describe our approach to Theorem \ref{riad}. In particular, we explain some of the difficulties involved in proving effective error bounds for $\spt(n)$. 

In \cite{R}, Rademacher established an exact formula for $p(n)$ as the absolutely convergent infinite sum
\begin{align}\label{rademacher}
p(n)=\frac{2\pi}{(24n-1)^{3/4}}\sum_{c=1}^\infty\frac{A_c(n)}{c}I_{3/2}\left(\frac{\pi\sqrt{24n-1}}{6c}\right),
\end{align}
where $I_{\nu}$ is the $I$-Bessel function, $A_c(n)$ is the Kloosterman-type sum
\begin{equation*}
A_c(n):=\sum_{\substack{d \shortmod c\\[1pt](d,c)=1}}e^{\pi is(d,c)}e^{-\frac{2\pi idn}{c}},
\end{equation*}
and $s(d,c)$ is the classical Dedekind sum
\begin{align*}
s(d,c):=\sum_{r=1}^{c-1}\frac{r}{c}\left(\frac{dr}{c}-\left\lfloor\frac{dr}{c}\right\rfloor-\frac{1}{2}\right).
\end{align*}

Recently, Ahlgren and Andersen \cite{AA} gave a Rademacher-type exact formula for $\spt(n)$ 
as the conditionally convergent infinite sum 
\begin{align}\label{AA}
\spt(n)=\frac{\pi}{6}(24n-1)^\frac{1}{4}\sum_{c=1}^\infty\frac{A_c(n)}{c}\left(I_{1/2}-I_{3/2}\right)\left(\frac{\pi\sqrt{24n-1}}{6c}\right).
\end{align}

In order to give an effective bound on the error term for $p(n)$, Lehmer \cite{L} truncated the absolutely convergent 
sum (\ref{rademacher}) and applied bounds for the Kloosterman-type sum $A_c(n)$. On the other hand, since the formula (\ref{AA}) is only 
conditionally convergent, bounding $\spt(n)$ is a much more delicate matter. In fact, to resolve the difficult problem of proving that 
(\ref{AA}) converges, Ahlgren and Andersen used advanced methods from the spectral theory of automorphic forms. 
To give an effective bound on the error term for $\spt(n)$, we will instead use   
different types of formulas for $p(n)$ and $\spt(n)$ which express these functions as traces of singular moduli.

To state these formulas, consider the weight $-2$ weakly holomorphic 
modular form for $\Gamma_0(6)$ defined by 
\begin{align*}
g(z):=\frac{1}{2}\frac{E_2(z)-2E_2(2z)-3E_2(3z)+6E_2(6z)}{\big(\eta(z)\eta(2z)\eta(3z)\eta(6z)\big)^2}, \quad z=x+iy \in \mathbb{H}.
\end{align*}
By applying the Maass weight-raising operator to $g(z)$, one gets the following weight zero weak Maass form for $\Gamma_0(6)$,
\begin{align*}
P(z):= -\left(\frac{1}{2\pi i} \frac{d}{dz} + \frac{1}{2\pi y}\right)g(z).
\end{align*}

Bruinier and Ono \cite{BrO} proved the following formula for $p(n)$.

\begin{thm}[Bruinier--Ono]
For all $n \geq 1$, we have 
\begin{equation}\label{boformula}
p(n)=\frac{1}{24n-1}\sum_{[Q]}P(\tau_Q),
\end{equation}
where the sum is over the $\Gamma_0(6)$ equivalence classes of discriminant $-24n+1$ positive definite, 
integral binary quadratic forms $Q=[a,b,c]$ such that $6 | a$ and $b \equiv 1 \pmod {12}$, 
and $\tau_Q$ is the Heegner point given by the root $Q\left(\tau_Q,1\right)=0$ in the 
complex upper half-plane $\mathbb{H}$.
\end{thm}

Similarly, consider the weight zero weakly holomorphic modular form for $\Gamma_0(6)$ defined by 
\begin{align}\label{fdefinition}
f(z):=\frac{1}{24}\frac{E_4(z)-4E_4(2z)-9E_4(3z)+36E_4(6z)}{\big(\eta(z)\eta(2z)\eta(3z)\eta(6z)\big)^2}.
\end{align}

Ahlgren and Andersen \cite{AA} proved the following analogue of \eqref{boformula} for $\spt(n)$. 

\begin{thm}[Ahlgren--Andersen]
For all $n \geq 1$, we have
\begin{equation}\label{AAspt}
\spt(n)=\frac{1}{12}\sum_{[Q]}(f(\tau_Q)-P(\tau_Q)\big).
\end{equation}
\end{thm}

Identities which express 
Fourier coefficients of weak Maass forms as traces of singular moduli 
have been used in many contexts to give strong asymptotic formulas. 
For example, Bringmann and Ono \cite{BriO} expressed $p(n)$ as a twisted trace of singular 
moduli by arithmetically reformulating Rademacher's exact formula (\ref{rademacher}) for $p(n)$. 
Folsom and the second author \cite{FM} then combined the Bringmann-Ono formula
with spectral methods and subconvexity bounds for quadratic twists of modular $L$--functions
to give an asymptotic formula for $p(n)$ with a power-saving error term. In particular, by calculating the main term in this 
asymptotic formula in terms of the truncated main term in Rademacher's exact formula for $p(n)$, these authors improved the exponent 
in Lehmer's bound \cite{L}. This exponent was further improved by Ahlgren and Andersen \cite{AA2}.

In the works \cite{M, M2, BBMS}, spectral methods and subconvexity bounds were again used 
to give asymptotic formulas with power-saving error terms 
for twisted traces of singular moduli. These results were applied to study a variety of arithmetic problems, including the distribution of 
$p(n)$ and $\spt(n)$, and the distribution of partition ranks. 
Note that although the constants in the error terms of these results are effective, it would be very difficult to actually give explicit numerical 
values for these constants because of the techniques involved in the proofs. However, there is an alternative approach 
which we now describe.

Using (\ref{boformula}), the formula (\ref{AAspt}) can be written as 
\begin{align}\label{reorganized}
\spt(n)=\frac{1}{12}S(n) - \frac{24n-1}{12}p(n),
\end{align}
where $S(n)$ is the trace of singular moduli for $f(z)$ given by
\begin{align*}
S(n):=\sum_{[Q]}f(\tau_Q).
\end{align*}

By applying Lehmer's effective error bounds for $p(n)$ in (\ref{reorganized}), we will reduce the proof of Theorem \ref{riad} to 
the following asymptotic formula for the trace $S(n)$ with an effective bound on the error term. 

\begin{theorem}\label{Sprop}
Let $\l(n):=\pi \sqrt{24n-1}/6$. Then for all $n \geq 1$, we have
\begin{align*}
S(n)=2\sqrt{3} e^{\l(n)} + E(n),
\end{align*}
where 
\begin{align*}
|E(n)| < (4.30 \times 10^{23}) 2^{q(n)} (24n-1)^2 e^{\l(n)/2}
\end{align*}
with 
\begin{align*}
q(n):=\frac{\log(24n-1)}{|\log(\log(24n-1))-1.1714|}.
\end{align*}
\end{theorem}

\vspace{0.10in}

Our proof of Theorem \ref{Sprop} is inspired by work of Dewar and Murty \cite{DM}, 
who used the formula (\ref{boformula}) to derive the Hardy-Ramanujan asymptotic formula for $p(n)$ 
without using the circle method. In order to give effective bounds, additional care must be taken. 
For instance, we must give effective bounds for the Fourier coefficients of $f(z)$.

\vspace{0.10in}

\textbf{Organization}. The paper is organized as follows. In Section \ref{heegner}, we review some facts regarding quadratic forms and Heegner points. 
In Section \ref{trace}, we prove Theorem \ref{Sprop}. 
In Section \ref{s2}, we prove Theorem \ref{riad}. 
In Section \ref{s3}, we prove Theorem \ref{better}. Finally, in Section \ref{s4}, we prove the remaining conjectures.

\vspace{0.10in}

\textbf{Acknowledgments}.
We thank Adrian Barquero-Sanchez, Sheng-Chi Liu, Karl Mahlburg, Ken Ono, Wei-Lun Tsai, and Matt Young for very helpful discussions regarding this work, 
and Michael Griffin and Lea Beneish for help computing values of $\spt(n)$. We also thank the referees for many detailed comments and corrections, 
leading to simplifications of some arguments, sharper estimates, and an improved exposition.

\section{Quadratic forms and Heegner points}\label{heegner}

Let $N \geq 1$ be a positive integer and $D < 0$ be a negative discriminant coprime to $N$. Let 
$\mathcal{Q}_{D,N}$ be the set of positive definite, integral binary quadratic forms 
$$Q(X,Y)=[a_Q,b_Q,c_Q](X,Y)=a_Q X^2+b_Q XY+c_Q Y^2$$ of discriminant $b_Q ^2-4a_Q c_Q= D < 0$ with $a_Q \equiv 0 \pmod N$. 
There is a (right) action of $\Gamma_0(N)$ on $\mathcal{Q}_{D,N}$ defined by 
\begin{align*}
Q=\left[a_Q, b_Q, c_Q\right] \mapsto Q \circ \sigma = \left[a_Q^{\sigma}, b_Q^{\sigma}, c_Q^{\sigma}\right],
\end{align*}
where  
for $\sigma = \begin{pmatrix} \alpha & \beta \\ \gamma & \delta \end{pmatrix} \in \Gamma_0(N)$ we have
\begin{align*}
a_Q^{\sigma}&= a_Q \alpha^2 + b_Q \alpha \gamma + c_Q \gamma^2 ,\\
b_Q^{\sigma}&=  2 a_Q \alpha \beta + b_Q(\alpha \delta + \beta \gamma) + 2 c_Q \gamma \delta ,\\
c_Q^{\sigma}&=  a_Q \beta^2 + b_Q \beta \delta +c_Q \delta^2.
\end{align*}  

Given a solution $r \pmod {2N}$ of $r^2 \equiv D \pmod {4N}$, we define the subset of forms 
\begin{align*}
\mathcal{Q}_{D,N,r}:=\left\{Q=\left[a_Q,b_Q,c_Q\right] \in \mathcal{Q}_{D,N}:~b_Q \equiv r \shortmod {2N}\right\}. 
\end{align*}
Then the group $\Gamma_0(N)$ also acts on $\mathcal{Q}_{D,N,r}$. The number of $\Gamma_0(N)$ equivalence classes in 
$\mathcal{Q}_{D,N,r}$ is given by the Hurwitz-Kronecker class number $H(D)$.

The preceding facts remain true if we restrict to the subset $\mathcal{Q}_{D,N}^{\textrm{prim}}$ of primitive forms in $\mathcal{Q}_{D,N}$; i.e., 
those forms with $$(a_Q, b_Q, c_Q)=1.$$ In this case, the number of $\Gamma_0(N)$ equivalence classes in 
$\mathcal{Q}_{D,N,r}^{\textrm{prim}}$ is given by the class number $h(D)$.

To each form $Q \in \mathcal{Q}_{D,N}$, we associate a Heegner point $\tau_Q$ which is the root of $Q\left(X,1\right)$ given by
\begin{align*}
\tau_Q=\frac{-b_Q+\sqrt{D}}{2a_Q} \in \mathbb{H}.
\end{align*}
The Heegner points $\tau_Q$ are compatible with the action of $\Gamma_0(N)$ in the sense that if $\sigma \in \Gamma_0(N)$, then 
\begin{align}\label{compatible}
\sigma (\tau_{Q}) = \tau_{Q \circ \sigma^{-1}}.
\end{align}

\section{Proof of Theorem \ref{Sprop}}\label{trace}

In this section, we prove Theorem \ref{Sprop}, which gives an asymptotic formula with an effective bound on the error term for the 
trace of the weight zero weakly holomorphic modular form $f(z)$ for $\Gamma_0(6)$ defined by (\ref{fdefinition}). This will be used crucially in the 
proof of Theorem \ref{riad}.

Let $D_n:=-24n+1$ for $n \in \Z^{+}$ and define the trace of $f(z)$ by 
\begin{align*}
S(n):=\sum_{[Q] \in \mathcal{Q}_{D_n, 6, 1}/\Gamma_0(6)}f(\tau_Q).
\end{align*}

First, we decompose $S(n)$ as a linear combination of traces involving primitive forms. Let 
$\Delta < 0$ be any discriminant with $\Delta \equiv 1 \pmod {24}$ and define the class polynomials 
\begin{align*}
H_n(X):=\prod_{[Q] \in \mathcal{Q}_{D_n, 6, 1}/\Gamma_0(6)}(X-f(\tau_Q))
\end{align*} 
and 
\begin{align*}
\widehat{H}_{\Delta}(X):=\prod_{[Q] \in \mathcal{Q}_{\Delta, 6, 1}^{\textrm{prim}}/\Gamma_0(6)}(X-f(\tau_Q)).
\end{align*} 
Let $\{W_{\ell}\}_{\ell|6}$ be the group of Atkin-Lehner involutions for $\Gamma_0(6)$. Since 
\begin{align}\label{ALid}
f|_0W_{\ell}=\lambda_{\ell} f
\end{align}
with $\lambda_{\ell}=1$ for $\ell=1,6$ and $\lambda_{\ell}= -1$ for $\ell=2,3$, 
then arguing exactly as in the proof of \cite[Lemma 3.7]{BOS}, we get the identity 
\begin{align}\label{relation}
H_n(X)=\prod_{\substack{u > 0 \\ u^2 | D_n}}\varepsilon(u)^{h(D_{n}/u^2)}\widehat{H}_{D_{n}/u^2}(\varepsilon(u)X),
\end{align}
where $\varepsilon(u)=1$ if $u \equiv \pm 1 \pmod{12}$ and $\varepsilon(u)=-1$ otherwise. Comparing terms on both sides of (\ref{relation}) 
yields the class number relation 
\begin{align*}
H(D_n)=\sum_{\substack{u > 0 \\ u^2 | D_n}}h(D_{n}/u^2)
\end{align*}
and the decomposition 
\begin{align}\label{Sdecomp}
S(n)=\sum_{\substack{u > 0 \\ u^2 | D_n}}\varepsilon(u)S_u(n),
\end{align}
where 
\begin{align*}
S_u(n):=\sum_{[Q] \in \mathcal{Q}_{D_n/u^2, 6, 1}^{\textrm{prim}}/\Gamma_0(6)}f(\tau_Q).
\end{align*}

Next, following \cite{DM} we express $S_u(n)$ as a trace involving primitive forms of level 1. 
The group $\Gamma_0(6)$ has index 12 in $SL_2(\Z)$. We 
choose the following 12 right coset representatives: 
\begin{align*}
\gamma_{\infty}&:=\begin{pmatrix} 1 & 0 \\ 0 & 1
\end{pmatrix},\\
\gamma_{1/3,r}&:=\begin{pmatrix}
1 & 0 \\ 3 & 1 
\end{pmatrix}
\hspace{-0.05in}
\begin{pmatrix}
1 & r \\ 0 & 1
\end{pmatrix}, \quad r=0,1; \\
\gamma_{1/2,s}&:=\begin{pmatrix}
1 & 1 \\ 2 & 3 
\end{pmatrix}
\hspace{-0.05in}
\begin{pmatrix}
1 & s \\ 0 & 1
\end{pmatrix}, \quad s=0,1,2; \\
\gamma_{0,t}&:=\begin{pmatrix}
0 & -1 \\ 1 & 0 
\end{pmatrix}
\hspace{-0.05in}
\begin{pmatrix}
1 & t \\ 0 & 1
\end{pmatrix}, \quad t=0,1,2, 3, 4, 5.
\end{align*}
We denote this set of coset representatives by $\mathbf{C}_6$. Each matrix $\gamma \in \mathbf{C}_6$ maps the cusp $\infty$ to one of 
the four cusps $\{\infty, 1/3, 1/2, 0\}$ of the modular curve $X_0(6)$, which have widths 1, 2, 3, and 6, respectively. In particular, we have 
$\gamma_{\infty}(\infty)=\infty$, $\gamma_{1/3,r}(\infty)=1/3$, $\gamma_{1/2, s}(\infty)=1/2$, and $\gamma_{0,t}(\infty)=0$.

Recall that a form $Q=[a_Q, b_Q, c_Q] \in \mathcal{Q}_{\Delta,1}$ is reduced if 
\begin{align*}
|b_Q| \leq a_Q \leq c_Q,
\end{align*}
and if either $|b_Q|=a_Q$ or $a_Q=c_Q$, then $b_Q \geq 0$. Let 
$\mathcal{Q}_{\Delta}$ denote a set of primitive, reduced forms representing the equivalence classes in 
$\mathcal{Q}_{\Delta, 1}^{\textrm{prim}}/SL_2(\Z)$. 
For each $Q \in \mathcal{Q}_{\Delta}$, there is a unique choice of coset representative $\gamma_Q \in \mathbf{C}_6$ 
such that 
\begin{align*}
[Q \circ \gamma_Q^{-1}] \in \mathcal{Q}_{\Delta, 6, 1}^{\textrm{prim}}/\Gamma_0(6).
\end{align*} 
This induces a bijection 
\begin{align}\label{bijection}
\mathcal{Q}_{\Delta} & \longrightarrow \mathcal{Q}_{\Delta, 6, 1}^{\textrm{prim}}/\Gamma_0(6)\\
Q &\longmapsto [Q \circ \gamma_Q^{-1}];\notag
\end{align}
see the Proposition on page 505 in \cite{GKZ}, or more concretely, \cite[Lemma 3]{DM}, 
where an explicit list of the matrices $\gamma_Q \in \mathbf{C}_6$ is given.

Using the bijection (\ref{bijection}) and the compatibility relation (\ref{compatible}) for Heegner points, the trace $S_u(n)$ 
can be expressed as 
\begin{align}\label{newtrace}
S_u(n)=\sum_{[Q] \in \mathcal{Q}_{D_n/u^2, 6, 1}^{\textrm{prim}}/\Gamma_0(6)}f(\tau_Q)=\sum_{Q \in \mathcal{Q}_{D_{n}/u^2}}f\left(\gamma_Q(\tau_Q)\right).
\end{align}
Therefore, to study the asymptotic distribution of $S_u(n)$, 
we need the Fourier expansion of $f(z)$ 
with respect to the matrices $\gamma_{\infty}, \gamma_{1/3,r}, \gamma_{1/2,s}$, and $\gamma_{0,t}$.

In \cite[Section 4]{AA}, Ahlgren and Andersen compute the Fourier expansion of $f(z)$ at the cusp $\infty$. The basic idea is as follows. 
The weakly holomorphic modular form $f(z)$ has a Fourier expansion of the form 
\begin{align*}
f(z)=e(-z) + b(0) + \sum_{m=1}^{\infty}b(m)e(mz), \quad e(z):=e^{2\pi i z}
\end{align*}
for some integers $b(m)$ for $m \geq 0$. One can construct a weight zero weak Maass form $f(z,s)$ for $\Gamma_0(6)$ with eigenvalue $s(1-s)$ 
whose analytic continuation at $s=1$ is a harmonic function on $\mathbb{H}$ with the Fourier expansion 
\begin{align*}
f(z,1)=e(-z) + a(0) + \sum_{m=1}^{\infty}\frac{a(m)}{\sqrt{m}}e(mz) - e(-\overline{z}) + \sum_{m=1}^{\infty}\frac{a(-m)}{\sqrt{m}} e(-m\overline{z}), 
\end{align*}
where 
\begin{align*}
a(0)& =4\pi^2\sum_{\ell|6}\frac{\mu(\ell)}{\ell}\sum_{\substack{0 < c \equiv 0 \mod (6/\ell) \\ (c,\ell)=1}}
\frac{S\left(-\overline{\ell}, 0;c\right)}{c^{2}},\\
\vspace{0.05in}
a(m)&=2\pi \sum_{\ell|6}\frac{\mu(\ell)}{\sqrt{\ell}}\sum_{\substack{0 < c \equiv 0 \mod (6/\ell) \\ (c,\ell)=1}}
\frac{S\left(-\overline{\ell}, m ;c\right)}{c} I_1\left(\frac{4\pi\sqrt{m}}{\sqrt{\ell}c}\right), \quad m \geq 1 \\
\vspace{0.05in}
a(-m)&=2\pi \sum_{\ell|6}\frac{\mu(\ell)}{\sqrt{\ell}}\sum_{\substack{0 < c \equiv 0 \mod (6/\ell) \\ (c,\ell)=1}}
\frac{S\left(-\overline{\ell}, -m ;c\right)}{c} J_1\left(\frac{4\pi\sqrt{m}}{\sqrt{\ell}c}\right), \quad m \geq 1.
\end{align*}
Here $\mu(\ell)$ is the M\"obius function, $S(a,b;c)$ is the Kloosterman sum 
\begin{align*}
S(a,b;c):=\sum_{\substack{d \shortmod c \\ (c,d)=1}}e\left(\frac{a\overline{d} + bd}{c}\right),
\end{align*}
and $I_1, J_1$ are the Bessel functions of order 1 (note that $\overline{d}$ is the multiplicative inverse of $d \pmod c$). 
From these Fourier expansions, one can see that the functions $f(z)$ and $f(z,1)$ 
have the same principal parts in the cusps $\{\infty, 1/3, 1/2, 0\}$, hence the 
function $f(z)-f(z,1)$ is bounded on the compact Riemann surface $X_0(6)$. Since a bounded harmonic function on a compact Riemann surface is constant, 
the function $f(z)-f(z,1)$ is constant.  

Now, using the Fourier expansions of $E_4(z)$ and $\eta(z)$, we use \texttt{SageMath} to compute 
\begin{align*}
f(z)=q^{-1} + 12 + 77 q + 376 q^2 + 1299q^3 + 4600q^4 + 12025q^5 + \cdots , \quad q:=e(z). 
\end{align*} 
In particular, $b(0)=12$. On the other hand, in Lemma \ref{fourierbound} we show by a direct calculation that $a(0)=12$. Since
$f(z)-f(z,1)$ is constant, we have
\begin{align*}
f(z)-f(z,1)=b(0)-a(0)=12-12=0.
\end{align*} 
Finally, since $f(z)=f(z,1)$, then by uniqueness of Fourier expansions we have $b(m)=m^{-1/2}a(m)$ for $m \geq 1$, 
$a(-1)=1$, and $a(-m)=0$ for $m \geq 2$.  

We next use the Fourier expansion 
\begin{align*}
f|_{0}\gamma_{\infty}(z) = e(-z) + 12
+ \sum_{m=1}^{\infty}\frac{a(m)}{\sqrt{m}}e(mz)
\end{align*}
to compute the Fourier expansion of $f(z)$ 
with respect to the matrices $\gamma_{\infty}, \gamma_{1/3,r}, \gamma_{1/2,s}$, and $\gamma_{0,t}$.

The Atkin-Lehner involutions for $\Gamma_0(6)$ are given by 
\begin{align*}
W_1=
  \begin{pmatrix}
    1 & 0 \\
    0 & 1
  \end{pmatrix}, \quad 
W_2=
  \frac{1}{\sqrt{2}}\begin{pmatrix}
    2 & -1 \\
    6 & -2
  \end{pmatrix}, \quad
W_3=
  \frac{1}{\sqrt{3}}\begin{pmatrix}
    3 & 1 \\
    6 & 3
 \end{pmatrix}, \quad 
 W_6=
  \frac{1}{\sqrt{6}}\begin{pmatrix}
    0 & -1 \\
   6 & 0
  \end{pmatrix}.
\end{align*}

For each $\ell|6$ and $v=6/\ell$, let $V_{\ell} = \sqrt{\ell} W_{\ell}$ and 
\begin{align*}
A_{\ell} = 
\begin{pmatrix}
\frac{1}{\textrm{width of the cusp $1/v$}} & 0 \\
    0 & 1
\end{pmatrix}.
\end{align*}
We have 
\small
\begin{table}[H]{\tabulinesep=1.35mm
\begin{tabu}{|c|c|c|c|c|}
\hline
cusp $1/v$ & $\infty \simeq 1/6$ & $1/3$ & $1/2$ & $0 \simeq 1$ \\ \hline
$\ell$ & $1$ & $2$ & $3$ & $6$ \\ \hline
$V_{\ell}$ & $\begin{pmatrix}
    1 & 0 \\
    0 & 1
  \end{pmatrix}$ & $\begin{pmatrix}
    2 & -1 \\
    6 & -2
  \end{pmatrix}$ & $\begin{pmatrix}
    3 & 1 \\
    6 & 3
 \end{pmatrix}$ & $\begin{pmatrix}
    0 & -1 \\
   6 & 0
  \end{pmatrix}$ \\ \hline
 $A_{\ell}$ & $\begin{pmatrix}
    1 & 0 \\
    0 & 1
  \end{pmatrix}$ & $\begin{pmatrix}
    1/2 & 0 \\
    0 & 1
  \end{pmatrix}$ & $\begin{pmatrix}
    1/3 & 0 \\
    0 & 1
 \end{pmatrix}$ & $\begin{pmatrix}
    1/6 & 0 \\
   0 & 1
  \end{pmatrix}$ \\ \hline
 $V_{\ell}A_{\ell}$ & $\begin{pmatrix}
    1 & 0 \\
    0 & 1
  \end{pmatrix}$ & $\begin{pmatrix}
    1 & -1 \\
    3 & -2
  \end{pmatrix}$ & $\begin{pmatrix}
    1 & 1 \\
    2 & 3
 \end{pmatrix}$ & $\begin{pmatrix}
    0 & -1 \\
   1 & 0
  \end{pmatrix}$ \\ \hline
\end{tabu}
}
\end{table}
\normalsize
Note that $V_{\ell}A_{\ell} \in SL_2(\mathbb{Z})$ and 
\begin{align*}
V_{\ell}A_{\ell}(\infty)=\frac{1}{v}.
\end{align*}
Let $\gamma \in SL_2(\Z)$ be any matrix such that $\gamma(\infty)=1/v$. Then 
\begin{align*} 
\left(V_{\ell}A_{\ell}\right)^{-1}(\gamma(\infty)) = \infty, 
\end{align*}
so that 
\begin{align*}
\left(V_{\ell}A_{\ell}\right)^{-1}\gamma \in \Gamma_{\infty}:=\left\{\begin{pmatrix} 1  & n \\ 0 & 1 \end{pmatrix}:~n \in \Z \right\}
\end{align*} 
where $\Gamma_\infty$ is the stabilizer of the cusp $\infty$. In particular,  
there is an integer $n \in \mathbb{Z}$ such that 
\begin{align*}
\gamma=V_{\ell}A_{\ell}
\begin{pmatrix}
    1 & n \\
    0 & 1
\end{pmatrix}.
\end{align*}
By solving for $n$ for each cusp, we have 
\begin{align*}
\gamma_{\infty} = V_1A_1, \quad
\gamma_{1/3,r} = V_2A_2 
\begin{pmatrix}
1 & r+1 \\
0 & 1
\end{pmatrix}, \quad 
\gamma_{1/2,s} & = V_3A_3
\begin{pmatrix}
    1 & s \\
    0 & 1
  \end{pmatrix}, \quad
  \gamma_{0,t}= V_6A_6
\begin{pmatrix}
    1 & t \\
    0 & 1
\end{pmatrix}.
\end{align*}

Now, by (\ref{ALid}) we have
$f(V_{\ell}z)=f(z)$ for $\ell=1,6$ and $f(V_{\ell}z)=-f(z)$ for $\ell=2,3$. Hence 
\begin{align*}
   f|_{0}\gamma_{\infty} (z) & = f(z),\\
   f|_{0}\gamma_{1/3,r} (z) & = 
f\left(V_2A_2\begin{pmatrix}
    1 & r+1 \\
    0 & 1
  \end{pmatrix} z\right) = f\left(V_2\left(\frac{z+r+1}{2}\right)\right) = -f\left(\frac{z+r+1}{2}\right),\\
f|_{0}\gamma_{1/2,s} (z) &= f\left(V_3A_3\begin{pmatrix}
    1 & s \\
    0 & 1
  \end{pmatrix} z\right) = f\left(V_3\left(\frac{z+s}{3}\right)\right) = - f\left(\frac{z+s}{3}\right),\\
 f|_{0}\gamma_{0,t} (z) &= f\left(V_6A_6\begin{pmatrix}
     1 & t \\
     0 & 1
   \end{pmatrix} z\right) = f\left(V_6\left(\frac{z+t}{6}\right)\right) = f\left(\frac{z+t}{6}\right).
\end{align*}

The Fourier expansion of $f(z)$ with respect to the matrices $\gamma_{1/3,r}, \gamma_{1/2,s}, \gamma_{0,t}$ can now 
be determined from the Fourier expansion at $\infty$ using these identities. In particular, if $\zeta_6:=e(1/6)$ is a primitive sixth root of unity, 
we have 
\begin{align*}
f|_{0}\gamma_{1/3,r}(z)  & = \zeta_6^{3r}e(-z/2) - 12 
+ \sum_{m=1}^{\infty}\zeta_6^{3+3m(r+1)}\frac{a(m)}{\sqrt{m}}e(mz/2),\\
f|_{0}\gamma_{1/2,s}(z) & = \zeta_6^{3-2s}e(-z/3) - 12 
+ \sum_{m=1}^{\infty}\zeta_6^{3+2ms}\frac{a(m)}{\sqrt{m}}e(mz/3),\\
f|_{0}\gamma_{0,t}(z) &= \zeta_6^{-t}e(-z/6) + 12 
+ \sum_{m=1}^{\infty}\zeta_6^{mt}\frac{a(m)}{\sqrt{m}}e(mz/6).
\end{align*}

Given a form $Q \in Q_\Delta$ and corresponding coset representative 
$\gamma_{Q} \in \mathbf{C}_6$, let $h_Q \in \{1,2,3,6\}$ be the width of the cusp 
$\gamma_{Q}(\infty)$, and let $\zeta_Q$ and $\phi_{m,Q}$ be the sixth roots of unity defined as follows:
\begin{table}[H]
{\tabulinesep=1.35mm
\begin{tabu}{|c|c|c|c|c|}
\hline
 cusp $\gamma_{Q}(\infty)$ & $\infty \simeq 1/6$ & $1/3$ & $1/2$ & $0 \simeq 1$ \\ \hline
 $\zeta_Q$ & $1$ & ${\zeta_6}^{3r}$ & ${\zeta_6}^{3-2s}$ & ${\zeta_6}^{-t}$ \\ \hline
 $\phi_{m,Q}$ & $1$ & ${\zeta_6}^{3+3m(r+1)}$ & ${\zeta_6}^{3+2ms}$ & ${\zeta_6}^{mt}$ \\ \hline
\end{tabu} 
}
\label{table:rootU}
\end{table}
Then we can write 
\begin{align}\label{formexpand}
f|_{0}\gamma_{Q}(z)  & = \zeta_Qe(-z/h_Q) + 12\mu(h_Q)
+ \sum_{m=1}^{\infty}\phi_{m,Q}\frac{a(m)}{\sqrt{m}}e(mz/h_Q).
\end{align}

In the following lemma we evaluate $a(0)$ and give effective bounds for the Fourier coefficients $a(m)$ for $m \geq 1$. 

\begin{lemma}\label{fourierbound} We have $a(0) = 12$ and
\begin{align*}
|a(m)| \leq C \sqrt{m}\exp(4\pi\sqrt{m}), \quad m \geq 1
\end{align*}
where 
\begin{align*}
C:=  8\sqrt{6}\pi^{3/2} + 16 \pi^2 \zeta^2(3/2).
\end{align*}     
\end{lemma}

\begin{proof} We first evaluate $a(0)$. Recall that 
\begin{align*}
a(0) =4\pi^2\sum_{\ell|6}\frac{\mu(\ell)}{\ell}\sum_{\substack{0 < c \equiv 0 \mod (6/\ell) \\ (c,\ell)=1}}
\frac{S\left(-\overline{\ell}, 0;c\right)}{c^{2}}.
\end{align*}
Since $(\ell,c)=1$, we can evaluate the Ramanujan sum as  
\begin{align*}
S\left(-\overline{\ell}, 0;c\right)=\sum_{\substack{d \shortmod c \\ (c,d)=1}}e\left(\frac{-\overline{\ell d}}{c}\right)
=\sum_{\substack{d \shortmod c \\ (c,d)=1}}e\left(\frac{\overline{\ell}d}{c}\right)=\mu(c)
\end{align*}
where the last equality follows from \cite[Equation (3.4)]{IK}. Hence 
\begin{align*}
a(0) =4\pi^2\sum_{\ell|6}\frac{\mu(\ell)}{\ell}\sum_{\substack{0 < c \equiv 0 \mod (6/\ell) \\ (c,\ell)=1}}
\frac{\mu(n)}{c^{2}}.
\end{align*}
Now, if $\ell=1$ we have 
\begin{align*}
\sum_{\substack{0 < c \equiv 0 \mod (6/\ell) \\ (c,\ell)=1}}\frac{\mu(n)}{c^{2}}&=\sum_{n=1}^{\infty}\frac{\mu(6n)}{(6n)^2}\\
&=\frac{1}{36}\sum_{\substack{n=1\\(n,6)=1}}^{\infty}\frac{\mu(n)}{n^2}\\
&=\frac{1}{36}\frac{1}{\zeta(2)}(1-2^{-2})^{-1}(1-3^{-2})^{-1}\\
&=\frac{1}{24}\frac{1}{\zeta(2)}.
\end{align*}
A similar calculation yields 
\begin{align*}
\sum_{\substack{0 < c \equiv 0 \mod (6/\ell) \\ (c,\ell)=1}}\frac{\mu(n)}{c^{2}}=
\frac{1}{\zeta(2)}\begin{cases}-1/6, & \ell = 2\\
-3/8, & \ell = 3\\ 
3/2, & \ell = 6.
\end{cases}
\end{align*}
Then using $\zeta(2)=\pi^2/6$ we get 
\begin{align*}
a(0)=24\left(\frac{1}{24} + \frac{1}{12} + \frac{1}{8} + \frac{1}{4} \right) =12.
\end{align*}

We next estimate $|a(m)|$ for $m \geq 1$. From the series (see \texttt{dlmf.nist.gov/10.25.2})
\begin{align*}
I_1(x)=\frac{x}{2}\sum_{k=0}^{\infty}\frac{(x^2/4)^k}{k!\Gamma(k+2)},
\end{align*}
we get  
\begin{align}\label{Ibound1}
|I_1(x)| \leq x \quad \textrm{for} \quad 0 < x < 1.
\end{align}
Also, using the asymptotic expansion (see \texttt{dlmf.nist.gov/10.40.1}) and the error bounds (see \texttt{dlmf.nist.gov/10.40.(ii)}), we get 
\begin{align}\label{Ibound2}
|I_1(x)|  \leq  \frac{1}{\sqrt{2\pi}}\frac{1}{\sqrt{x}}\exp(x) \quad \textrm{for} \quad x \geq 1.
\end{align}

Let $M=4\pi\sqrt{m}/\sqrt{\ell}$. Then using the Weil bound 
\begin{align*}
|S(a,b;c)| \leq \tau(c)(a,b,c)^{1/2}c^{1/2}
\end{align*}
where $\tau(c)$ is the number of divisors of $c$, and the estimates (\ref{Ibound1}) and (\ref{Ibound2}), we get 
\begin{align*}
|a(m)| & \leq 2\pi \sum_{\ell | 6}\frac{|\mu(\ell)|}{\sqrt{\ell}}
\sum_{\substack{0 < c \equiv 0 \mod (6/\ell) \\ (c,\ell)=1}}\frac{\left|S\left(-\overline{\ell}, m ;c\right)\right|}{c} |I_1(M/c)| \\
& \leq \frac{1}{\sqrt{2}}m^{-1/4} S_1 +  8 \pi^{2} m^{1/2} S_2,
\end{align*}
where
\begin{align*}
S_1:=\sum_{\ell |6} \ell^{-1/4} \sum_{\substack{0 < c \leq  M \\ (c,\ell)=1}} \tau(c)\exp(4\pi\sqrt{m}/\sqrt{\ell}c) 
\end{align*}
and 
\begin{align*} 
S_2:=\sum_{\ell|6} \ell^{-1} \sum_{\substack{c > M \\(c,\ell)=1}} \frac{\tau(c)}{c^{3/2}}.
\end{align*}

Using the bound (see \cite{NR})
\begin{align*}
\tau(n) \leq n^{1.538 \frac{\log(2)}{\log(\log(n))}}, \quad n \geq 2
\end{align*}
which implies that $\tau(n) \leq \sqrt{3}n^{1/2}$ for $n \geq 1$, we get
\begin{align*}
|S_1| & \leq \sqrt{3} \exp\left(4\pi\sqrt{m}\right) \sum_{\ell|6} \ell^{-1/4} \sum_{0 < c \leq M} c^{1/2}\\  
& \leq 2 \sqrt{3} (4\pi)^{3/2}m^{3/4} \exp\left(4\pi\sqrt{m}\right). 
\end{align*}
Also,  
\begin{align*}
|S_2| \leq  2 \sum_{c=1}^{\infty}\frac{\tau(c)}{c^{3/2}}= 2 \zeta^2(3/2).
\end{align*}
Then combining estimates yields 
\begin{align*}
|a(m)| \leq C  m^{1/2} \exp\left(4\pi\sqrt{m}\right), \quad m \geq 1
\end{align*}
where
\begin{align*}
C:=  8\sqrt{6}\pi^{3/2} + 16 \pi^2 \zeta^2(3/2). 
\end{align*}
\end{proof}

We are now in position to prove Theorem \ref{Sprop}, which we restate for the convenience of the reader.

\begin{theorem}
For all $n \geq 1$, we have 
\begin{align*}
S(n)= 2\sqrt{3} \exp(\pi \sqrt{24n-1}/6) + E(n),
\end{align*}
where 
\begin{align*}
|E(n)| < (4.30 \times 10^{23}) 2^{q(n)} (24n-1)^2 \exp(\pi\sqrt{24n-1}/12)
\end{align*}
with 
\begin{align*}
q(n):=\frac{\log(24n-1)}{|\log(\log(24n-1))-1.1714|}.
\end{align*}
\end{theorem}

\begin{proof} By (\ref{Sdecomp}), (\ref{newtrace}) and (\ref{formexpand}) we have  
\begin{align*}
S(n) & = \sum_{\substack{u > 0 \\ u^2 | D_n}}\varepsilon(u)S_u(n)\\ 
& =  \sum_{\substack{u > 0 \\ u^2 | D_n}}\varepsilon(u)\sum_{Q \in \mathcal{Q}_{D_{n}/u^2}}f|_{0}\gamma_Q(\tau_Q)\\
&= \sum_{\substack{u > 0 \\ u^2 | D_n}}\varepsilon(u)\sum_{Q \in \mathcal{Q}_{D_{n}/u^2}} 
\zeta_Q e(-\tau_{Q}/h_Q) + E_1(n), 
\end{align*}
where
\begin{align*}
E_1(n):= 12\mu(h_Q)\sum_{\substack{u > 0 \\ u^2 | D_n}}\varepsilon(u)h(D_n/u^2)
+ \sum_{m=1}^{\infty}\frac{a(m)}{\sqrt{m}}\sum_{\substack{u > 0 \\ u^2 | D_n}}\varepsilon(u)\sum_{Q \in \mathcal{Q}_{D_n/u^2}}\phi_{m,Q}e(m\tau_Q/h_Q).
\end{align*}

We have   
\begin{align*}
\left|12\mu(h_Q)\sum_{\substack{u > 0 \\ u^2 | D_n}}\varepsilon(u)h(D_n/u^2)\right| \leq 12 \sum_{\substack{u > 0 \\ u^2 | D_n}}h(D_n/u^2) = 12 H(D_n).
\end{align*}

Next, observe that 
\begin{align*}
e\left(m\tau_{Q}/h_Q\right)=\zeta_{2a_Q h_Q}^{-b_{Q}m}\exp\left(-\frac{\pi m \sqrt{|D_n|/u^2}}{a_Q h_Q} \right),
\end{align*}
where $\zeta_{2a_Q h_Q}$ is the primitive $2a_Qh_Q$-th root of unity 
\begin{align*}
\zeta_{2a_Q h_Q}:=e\left(\frac{1}{2a_Q h_Q}\right).
\end{align*}
Since $Q \in \mathcal{Q}_{D_n/u^2}$ is reduced, the corresponding Heegner point $\tau_Q$ lies in the standard fundamental domain 
$\mathcal{F}$ for $SL_2(\Z)$. In particular, we have 
\begin{align*}
\textrm{Im}(\tau_Q)=\frac{\sqrt{|D_n|/u^2}}{2a_Q} \geq \sqrt{3}/2, 
\end{align*}
which implies that 
\begin{align*}
a_Q \leq \frac{\sqrt{|D_n|/u^2}}{\sqrt{3}}.
\end{align*}
Since $h_Q \leq 6$, we have
\begin{align}\label{Minkowski}
-\frac{\pi m \sqrt{|D_n|/u^2}}{a_Q h_Q} \leq - \frac{\pi m}{2\sqrt{3}}.
\end{align}
Then using (\ref{Minkowski}) and Lemma \ref{fourierbound}, we get   
\begin{align*}
\left|\sum_{m=1}^{\infty}\frac{a(m)}{\sqrt{m}}\sum_{\substack{u > 0 \\ u^2 | D_n}}\varepsilon(u)\sum_{Q \in \mathcal{Q}_{D_n/u^2}}\phi_{m,Q}e(m\tau_Q/h_Q)\right|
& \leq \sum_{m=1}^{\infty}\frac{|a(m)|}{\sqrt{m}}\sum_{\substack{u > 0 \\ u^2 | D_n}}\sum_{Q \in \mathcal{Q}_{D_n/u^2}}|e(m\tau_Q/h_Q)|\\
& \leq C H(D_n) \sum_{m=1}^{\infty}\exp(4\pi \sqrt{m} -\pi m/2\sqrt{3}).
\end{align*}

Combining the preceding estimates yields
\begin{align*}
|E_1(n)| \leq 12 H(D_n) + C H(D_n) \sum_{m=1}^{\infty}\exp(4\pi \sqrt{m} -\pi m/2\sqrt{3}).
\end{align*}

To estimate the infinite sum, we write 
\begin{align*}
4\pi\sqrt{m}- \pi m/2\sqrt{3} = -m\left(\frac{\pi }{2 \sqrt{3}} -\frac{4\pi}{\sqrt{m}}\right),
\end{align*}
and observe that 
\begin{align*}
\frac{\pi}{2\sqrt{3}} -\frac{4\pi}{\sqrt{m}} > 0 \quad \Longleftrightarrow \quad m \geq 193,
\end{align*}
in which case we have 
\begin{align*}
-m\left(\frac{\pi}{2\sqrt{3}} -\frac{4\pi}{\sqrt{m}}\right) \leq -m \left(\frac{\pi}{2\sqrt{3}} -\frac{4\pi}{\sqrt{193}}\right).
\end{align*}
We then split the infinite sum into appropriate ranges and use the preceding bound to get
\begin{align*}
&\sum_{m=1}^{\infty}\exp\left(4\pi\sqrt{m}- \pi m/2 \sqrt{3}\right)\\
& \leq \sum_{m=1}^{192}\exp\left(4\pi\sqrt{m}-\pi m /2 \sqrt{3}\right) 
+ \sum_{m=193}^{\infty}\exp\left(-m\left(\frac{\pi }{2 \sqrt{3}} -\frac{4\pi}{\sqrt{193}}\right)\right).
\end{align*}
A calculation shows that
\begin{align*}
 \sum_{m=1}^{192}\exp\left(4\pi\sqrt{m}-\pi m/2\sqrt{3}\right) < 2.08 \times 10^{20}
\end{align*}
and 
\begin{align*}
\sum_{m=193}^{\infty}\exp\left(-m\left(\frac{\pi }{2 \sqrt{3}} -\frac{4\pi}{\sqrt{193}}\right)\right) < 426.
\end{align*}

We have now shown that 
\begin{align*}
|E_1(n)|  \leq C_1 H(D_n),
\end{align*}
where 
\begin{align*}
C_1:= 12 +  C \left[2.08 \times 10^{20} + 426\right] < 2.47 \times 10^{23}.
\end{align*}

It remains to analyze the main term. Write the main term as 
\begin{align*}
\sum_{\substack{u > 0 \\ u^2 | D_n}}\varepsilon(u)\sum_{Q \in \mathcal{Q}_{D_{n}/u^2}} 
\zeta_Q e(-\tau_{Q}/h_Q)
= \sum_{Q \in \mathcal{Q}_{D_{n}}} \zeta_Q e(-\tau_{Q}/h_Q) + E_2(n),
\end{align*}
where
\begin{align*}
E_2(n):=\sum_{\substack{u \geq 2 \\ u^2 | D_n}}\varepsilon(u)\sum_{Q \in \mathcal{Q}_{D_{n}/u^2}} 
\zeta_Q e(-\tau_{Q}/h_Q).
\end{align*}

Observe that for any form $Q=[a_Q,b_Q, c_Q] \in \mathcal{Q}_{D_n/u^2}$, we have 
\begin{align}\label{six}
a_Qh_Q \equiv 0 \shortmod 6
\end{align}
and 
\begin{align}\label{difference}
e(-\tau_{Q}/h_Q) = \zeta_{2a_Qh_Q}^{b_Q} \exp\left(\frac{\pi \sqrt{|D_n|/u^2}}{a_Qh_Q}\right).
\end{align}

Now, by \cite[(4.2)]{DM} there are exactly 4 forms $Q \in \mathcal{Q}_{D_n}$ with $a_Qh_Q = 6$, and these are 
given by 
\begin{align*}
Q_1=[1,1,6n], \quad Q_2=[2,1,3n], \quad Q_3=[3,1,2n], \quad Q_4=[6,1,n]. 
\end{align*}
Moreover, the corresponding coset representatives $\gamma_{Q_{i}} \in \mathbf{C}_6$ such that 
$$[Q_i \circ \gamma_{Q_i}^{-1}] \in \mathcal{Q}_{D_n,6,1}^{\textrm{prim}}/\Gamma_0(6)$$ are 
given by 
\begin{align*}
\gamma_{Q_1}=\gamma_{0,1}, \quad \gamma_{Q_2}=\gamma_{1/2, -1}, \quad \gamma_{Q_3}=\gamma_{1/3,0}, \quad \gamma_{Q_4}=\gamma_{\infty}.
\end{align*}

Write 
\begin{align*} 
 \sum_{Q \in \mathcal{Q}_{D_n}} \zeta_Q e(-\tau_{Q}/h_Q)
= \sum_{i=1}^{4} \zeta_{Q_i} e(-\tau_{Q_{i}}/h_{Q_{i}}) + E_3(n),
\end{align*}
where 
\begin{align*}
E_3(n):= \sum_{\substack{Q \in \mathcal{Q}_{D_n}\\ Q \neq Q_i}} \zeta_Q e(-\tau_{Q}/h_Q).
\end{align*}
By (\ref{six}) we have $a_Qh_Q \geq 12$ for all $Q \neq Q_i$, hence using (\ref{difference}) we get
\begin{align*}
|E_3(n)| & \leq \sum_{\substack{Q \in \mathcal{Q}_{D_n} \\ Q \neq Q_i}} \exp(\pi \sqrt{|D_n|}/a_Qh_Q)\\
& \leq h(D_n)\exp(\pi \sqrt{|D_n|}/12).
\end{align*}

Similarly, by (\ref{six}) we have $a_Qh_Q \geq 6$ for all $Q \in \mathcal{Q}_{D_n/u^2}$, hence for $u \geq 2$ we have
\begin{align*}
\frac{\sqrt{|D_n|/u^2}}{a_Qh_Q} \leq \sqrt{|D_n|}/12.
\end{align*}
Then by (\ref{difference}) we have 
\begin{align*}
|E_2(n)| & \leq  \sum_{\substack{u \geq 2 \\ u^2 | D_n}}\sum_{Q \in \mathcal{Q}_{D_{n}/u^2}}\exp\left(\frac{\pi \sqrt{|D_n|/u^2}}{a_Qh_Q}\right)\\
& \leq H(D_n)\exp(\pi \sqrt{|D_n|}/12).
\end{align*}

Since $a_{Q_i}h_{Q_i}=6$ and $b_{Q_i}=1$ for $i=1,2,3,4$, using (\ref{difference}) we get 
\begin{align*}
\sum_{i=1}^{4} \zeta_{Q_i} e(-\tau_{Q_{i}}/h_{Q_{i}}) 
= \exp(\pi i/6) \sum_{i=1}^{4} \zeta_{Q_i} \cdot \exp(\pi \sqrt{|D_n|}/6).
\end{align*}
Also, from the Fourier expansion of $f(z)$ with respect to $\gamma_{0,1}, \gamma_{1/2,-1}, \gamma_{1/3,0}$, and $\gamma_{\infty}$ given previously,
we have 
\begin{align*}
\zeta_{Q_1}=\zeta_6^{-1}, \quad \zeta_{Q_2}=\zeta_6^{3-2(-1)}, \quad \zeta_{Q_3}=\zeta_6^{0}, \quad \zeta_{Q_4}=1.
\end{align*}
Hence
\begin{align*}
 \exp(\pi i /6) \sum_{i=1}^{4} \zeta_{Q_i} = 2\sqrt{3}.
\end{align*}

By combining the preceding results, we get 
\begin{align*}
S(n)= 2\sqrt{3} \exp(\pi \sqrt{|D_n|}/6) + E(n),
\end{align*}
where $E(n):=E_1(n) + E_2(n) + E_3(n)$ with 
\begin{align*}
|E(n)| &\leq |E_1(n)| + |E_2(n)| + |E_3(n)|\\
& < 2H(D_n)\exp(\pi\sqrt{|D_n|}/12) + (2.47 \times 10^{23})H(D_n)\\
& < (2.48 \times 10^{23})H(D_n)\exp(\pi\sqrt{|D_n|}/12).
\end{align*}

To complete the proof, we require only a crude effective upper bound for the Hurwitz-Kronecker 
class number $H(D_n)$. 

Write $D_n=d_{n}f_n^2$ with $d_{n} < 0$ a fundamental discriminant and $f_n \in \Z^{+}$.
Then we have the class number relation 
\begin{align}\label{classrelation}
H(D_n)=\sum_{\substack{u > 0\\ u^2 | D_n}}h(D_n/u^2)=\sum_{\substack{u > 0 \\ u | f_n}}h(u^2d_n).
\end{align}
Inserting the formula (see e.g. \cite[p. 233]{Co})
\begin{align*}
h(u^2d_n)=u\prod_{p | u}\left(1-\frac{\chi_{d_n}(p)}{p}\right)h(d_n)
\end{align*}
into (\ref{classrelation}) yields 
\begin{align*}
H(D_n)=\delta(n)h(d_n),
\end{align*}
where
\begin{align*}
\delta(n):=\sum_{\substack{u > 0 \\ u | f_n}}u\prod_{p | u}\left(1-\frac{\chi_{d_n}(p)}{p}\right).
\end{align*}
Now, a simple estimate yields
\begin{align*}
|\delta(n)| \leq \sqrt{|D_n|}\tau(|D_n|)2^{\omega(|D_n|)}
\end{align*}
where $\omega(|D_n|)$ is the number of prime divisors of $|D_n|$. We have 
\begin{align*}
\tau(|D_n|) < \sqrt{3} \sqrt{|D_n|}, 
\end{align*}
and by \cite[Th\'eor\`eme 13]{Ro} we have 
\begin{align*}
\omega(|D_n|) \leq \max\left\{1, \frac{\log(|D_n|)}{\log(\log(|D_n|))-1.1714}\right\} \leq \frac{\log(|D_n|)}{|\log(\log(|D_n|))-1.1714|}=:q(n).
\end{align*}
Hence 
\begin{align*}
|\delta(n)| \leq  \sqrt{3}2^{q(n)} |D_n|.
\end{align*}
Using the class number formula 
\begin{align*}
h(d_n)=\frac{\sqrt{|d_n|}}{\pi}L(\chi_{d_n},1)
\end{align*}
and the evaluation
\begin{align*}
L(\chi_{d_n},1)=-\frac{\pi}{|d_n|^{3/2}}\sum_{t=1}^{|d_n|-1}\chi_{d_n}(t)t,
\end{align*}
another simple estimate yields
\begin{align*}
h(d_n) \leq |d_n|.
\end{align*}
Then combining the preceding estimates gives 
\begin{align*}
H(D_n) \leq  \sqrt{3}2^{q(n)}|D_n|^2.
\end{align*}

Finally, using the class number bound we get 
\begin{align*}
|E(n)| <  (4.30 \times 10^{23}) 2^{q(n)} |D_n|^2 \exp(\pi\sqrt{|D_n|}/12).
\end{align*}
This completes the proof.

\end{proof}

\section{Proof of Theorem \ref{riad}}\label{s2}

In this section, we prove Theorem \ref{riad}. We will require an asymptotic formula for $p(n)$ with an effective bound on the error term 
due to Lehmer \cite{L}. For convenience, define
\begin{equation*}
\l(n):=\frac{\pi}{6}\sqrt{24n-1}.
\end{equation*}
Inspired by the Hardy-Ramanujan asymptotic for $p(n)$, Rademacher \cite{R} obtained the exact formula 
\begin{equation*}
p(n)=\frac{2\pi}{(24n-1)^{3/4}}\sum_{c=1}^\infty\frac{A_c(n)}{c}I_{3/2}\left(\frac{\l(n)}{c}\right),
\end{equation*}
where $A_c(n)$ is the Kloosterman-type sum
\begin{equation*}
A_c(n):=\sum_{\substack{d \shortmod c\\[1pt](d,c)=1}}e^{\pi is(d,c)}e^{-\frac{2\pi idn}{c}}
\end{equation*}
and $s(d,c)$ is the Dedekind sum
\begin{equation*}
s(d,c):=\sum_{r=1}^{c-1}\frac{r}{c}\left(\frac{dr}{c}-\left\lfloor\frac{dr}{c}\right\rfloor-\frac{1}{2}\right).
\end{equation*}

Using Rademacher's formula, Lehmer \cite{L} proved the following result.

\begin{theorem}[Lehmer]\label{Lehmer}
For all $n \geq 1$, we have 
\begin{equation*}
p(n)=\frac{\sqrt{12}}{24n-1}\sum_{c=1}^N\frac{A_c(n)}{\sqrt{c}}\left\{\left(1-\frac{c}{\l(n)}\right)e^{\l(n)/c}+\left(1+\frac{c}{\l(n)}\right)e^{-\l(n)/c}\right\} 
+ R_2(n,N),
\end{equation*}
where
\begin{equation*}
\left|R_2(n, N)\right|<\frac{N^{-2/3}\pi^2}{\sqrt{3}}\left\{\frac{N^3}{2\l(n)^3}\left(e^{\l(n)/N}-e^{-\l(n)/N}\right)+\frac{1}{6}-\frac{N^2}{\l(n)^2}\right\}.
\end{equation*}
\end{theorem}

We first use Theorem \ref{Lehmer} to deduce the following effective bound.

\begin{lemma}\label{plemma} For all $n \geq 1$, we have 
\begin{align*}
p(n)=  \frac{2\sqrt{3}}{24n-1}\left(1-\frac{1}{\l(n)}\right)e^{\l(n)} + E_p(n),
\end{align*}
where 
\begin{align*}
\left|E_p(n)\right| \leq (1313)e^{\l(n)/2}.
\end{align*}
\end{lemma}

\begin{proof} Using the identity 
\begin{align}\label{Iidentity2}
I_{3/2}(x)=\frac{1}{2}\sqrt{\frac{2}{\pi x}}\left[\left(1-\frac{1}{x}\right)e^{x} + \left(1 + \frac{1}{x}\right)e^{-x}\right],
\end{align}
we may write Theorem \ref{Lehmer} (with the choice $N=2$) as   
\begin{align}\label{pdecomp}
p(n)=\frac{2\pi}{(24n-1)^{3/4}}\sum_{c=1}^{2}\frac{A_c(n)}{c}I_{3/2}\left(\frac{\l(n)}{c}\right) + R_2(n,2),
\end{align}
where 
\begin{align*}
|R_2(n,2)| < \frac{\pi^2}{\sqrt{3}2^{2/3}}\left[\left(\frac{2}{\l(n)}\right)^3\left\{\frac{e^{\l(n)/2}-e^{-\l(n)/2)}}{2}\right\} 
+ 1/6 -\left(\frac{2}{\l(n)}\right)^2\right].
\end{align*}
Now, using (\ref{Iidentity2}) and (\ref{pdecomp}) we get
\begin{align*}
p(n) &= \frac{2\pi}{(24n-1)^{3/4}}I_{3/2}\left(\l(n)\right) + \frac{2\pi}{(24n-1)^{3/4}}\frac{A_2(n)}{2}I_{3/2}\left(\frac{\l(n)}{2}\right) + R_2(n,2)\\
&=  \frac{2\sqrt{3}}{24n-1}\left(1-\frac{1}{\l(n)}\right)e^{\l(n)} + E_p(n),
\end{align*}
where 
\begin{align*}
E_p(n):=  \frac{2\pi}{(24n-1)^{3/4}}\left[\frac{1}{2}\sqrt{\frac{2}{\pi \l(n)}}\left(1 + \frac{1}{\l(n)}\right)e^{-\l(n)} 
+ \frac{A_2(n)}{2}I_{3/2}\left(\frac{\l(n)}{2}\right)\right] + R_2(n,2).
\end{align*}
Using (\ref{Iidentity2}) we have the bound 
\begin{align}\label{halfbound}
I_{3/2}(x) < x^{-1/2}e^{x}, \quad x \geq 1.
\end{align}
Then an estimate using the trivial bound
\begin{align*}
|A_c(n)| < c
\end{align*}
and (\ref{halfbound}) yields
\begin{align*}
\left| \frac{2\pi}{(24n-1)^{3/4}}\frac{A_2(n)}{2}I_{3/2}\left(\frac{\l(n)}{2}\right)\right| \leq e^{\l(n)/2}.
\end{align*}
Similarly, two straightforward estimates yield 
\begin{align*}
\left|\frac{2\pi}{(24n-1)^{3/4}} \frac{1}{2}\sqrt{\frac{2}{\pi \l(n)}}\left(1 + \frac{1}{\l(n)}\right)e^{-\l(n)}\right| \leq 16 e^{\l(n)/2}
\end{align*}
and
\begin{align*}
|R_2(n,2)| < (1296) e^{\l(n)/2}.
\end{align*}
Hence 
\begin{align*}
|E_p(n)| \leq (1313) e^{\l(n)/2}.
\end{align*}
\end{proof}

\subsection{Proof of Theorem \ref{riad}} 
Using (\ref{boformula}), the formula (\ref{AAspt}) can be written as 
\begin{align}\label{sptformula}
\textrm{spt}(n)=\frac{1}{12}\left[S(n) -(24n-1)p(n)\right].
\end{align}
Then using (\ref{sptformula}), Theorem \ref{Sprop}, and Lemma \ref{plemma}, a straightforward calculation yields
\begin{align*}
\textrm{spt}(n) &= \frac{\sqrt{3}}{\pi \sqrt{24n-1}}e^{\l(n)} + E_s(n),
\end{align*}
where the error term 
\begin{align*}
E_s(n) & := \frac{E(n)}{12} -\frac{24n-1}{12}E_p(n)
\end{align*}
satisfies the bound 
\begin{align*}
|E_s(n)| < (3.59 \times 10^{22})2^{q(n)}(24n-1)^2e^{\l(n)/2}.
\end{align*}
\qed

\vspace{0.10in}

As pointed out by Bessenrodt and Ono \cite{BO}, it is straightforward to obtain from Theorem \ref{Lehmer} that
\begin{equation*}
\frac{\sqrt{3}}{12n}\left(1-\frac{1}{\sqrt{n}}\right)e^{\l(n)}<p(n)<\frac{\sqrt{3}}{12n}\left(1+\frac{1}{\sqrt{n}}\right)e^{\l(n)}
\end{equation*}
for all $n \geq 1$.  

We will use Theorem \ref{riad} to prove the following analogous statement for $\spt(n)$, 
where $\sqrt{n}$ is replaced by any positive integral power of $n$.

\begin{theorem}\label{bounds}
For each $\alpha \in \Z^+$ and $k \in \Z^+$, there is an effectively computable positive integer $B_k(\alpha)$ such that for all $n \geq B_k(\alpha)$, 
we have 
\begin{align*}
\frac{\sqrt{3}}{\pi\sqrt{24n-1}}\left(1-\frac{1}{\alpha n^k}\right)e^{\l(n)} < \spt(n) < \frac{\sqrt{3}}{\pi\sqrt{24n-1}}\left(1+\frac{1}{\alpha n^k}\right)e^{\l(n)}.
\end{align*}
\end{theorem}

\begin{proof} By Theorem \ref{riad} we have the bounds 
\begin{equation*}
\frac{\sqrt{3}}{\pi\sqrt{24n-1}}e^{\l(n)}-\left|E_s(n)\right|<\spt(n)<\frac{\sqrt{3}}{\pi\sqrt{24n-1}}e^{\l(n)}+\left|E_s(n)\right|,
\end{equation*}
where
\begin{equation*}
\left|E_s(n)\right| < (3.59 \times 10^{22})2^{q(n)}(24n-1)^2e^{\l(n)/2}.
\end{equation*}
Clearly, there is an effectively computable positive integer $B_{k}(\alpha)$ such that the inequality  
\begin{align*}
 (3.59 \times 10^{22})2^{q(n)}(24n-1)^2e^{\l(n)/2} < \frac{\sqrt{3}}{\pi\sqrt{24n-1}} \frac{1}{\alpha n^k}e^{\l(n)}
\end{align*}
holds for all $n \geq B_{k}(\alpha)$. For instance, if $\alpha=k=1$ then $B_{1}(1)=5729$. 
This completes the proof.
\end{proof}

\section{Proof of Theorem \ref{better}}\label{s3}

By Theorem \ref{riad} and Lemma \ref{plemma}, we may write 
\begin{align}\label{eqn1}
\textrm{spt}(n)&=\alpha(n)e^{\l(n)} + E_s(n)
\end{align}
and 
\begin{align}\label{eqn2}
p(n)&=\beta(n)e^{\l(n)} + E_p(n),
\end{align}
where 
\begin{align*}
\alpha(n):=\frac{\sqrt{3}}{\pi \sqrt{24n-1}}, \quad \beta(n):=\frac{2\sqrt{3}}{24n-1}\left(1 - \frac{6}{\pi \sqrt{24n-1}}\right).
\end{align*}
Also, for $\epsilon > 0$ we define 
\begin{align*}
\gamma(n):=\frac{\sqrt{6}}{\pi}\sqrt{n}, \quad \gamma(n,\epsilon):=\left(\frac{\sqrt{6}}{\pi} + \epsilon\right)\sqrt{n}.
\end{align*}
We must prove that there exists an effectively computable positive constant $N(\epsilon) > 0$ such that for all 
$n \geq N(\epsilon)$, we have
\begin{align}\label{chenbound2}
\gamma(n) p(n) < \textrm{spt}(n) < \gamma(n, \epsilon)p(n).
\end{align}

First, using (\ref{eqn1}) and (\ref{eqn2}) we find that 
the lower bound in (\ref{chenbound2}) is equivalent to 
\begin{align}\label{lower}
c_1(n)e^{\l(n)} > \gamma(n)E_p(n) - E_s(n), 
\end{align}
where $c_1(n):=\alpha(n)-\beta(n)\gamma(n)$. Now, the error bounds in Theorem \ref{riad} and Lemma \ref{plemma} imply that 
\begin{align*}
\left|\gamma(n)E_p(n) - E_s(n)\right| \leq c_2(n)e^{\l(n)/2},
\end{align*}
where 
\begin{align*}
c_2(n):=(1313)\gamma(n) + (3.59 \times 10^{22})2^{q(n)}(24n-1)^2.
\end{align*}
Then noting that $c_1(n) > 0$ for all $n \geq 1$, we find that (\ref{lower}) 
is implied by the bound
\begin{align*}
e^{\l(n)/2} > c_3(n):=\frac{c_2(n)}{c_1(n)}, 
\end{align*}
or equivalently, the bound
\begin{align}\label{lower2}
n > \frac{1}{24}\left[\left(\frac{12}{\pi}\log(c_3(n))\right)^2 + 1\right].
\end{align}
A calculation shows that (\ref{lower2}) holds for all $n \geq N:=5310$. 

Similarly, using (\ref{eqn1}) and (\ref{eqn2}) we find that 
the upper bound in (\ref{chenbound2}) is equivalent to 
\begin{align}\label{upper}
c_4(n, \epsilon)e^{\l(n)} > E_s(n) - \gamma(n,\epsilon)E_p(n), 
\end{align}
where $c_4(n,\epsilon):=\beta(n)\gamma(n,\epsilon)-\alpha(n)$. The error bounds in Theorem \ref{riad} and Lemma \ref{plemma} imply that 
\begin{align*}
|E_s(n) - \gamma(n, \epsilon)E_p(n)| \leq c_5(n, \epsilon)e^{\l(n)/2},
\end{align*}
where 
\begin{align*}
c_5(n, \epsilon):=(1313)\gamma(n, \epsilon) + (3.59 \times 10^{22})2^{q(n)}(24n-1)^2.
\end{align*} 
Moreover, there exists an effectively computable positive constant $N_1(\epsilon) > 0$ such that $c_4(n, \epsilon) > 0$ for all $n \geq N_1(\epsilon)$. 
Then arguing as above, we find that if $n \geq N_1(\epsilon)$, the bound (\ref{upper}) is implied by the bound
\begin{align}\label{upper2}
n > \frac{1}{24}\left[\left(\frac{12}{\pi}\log(c_6(n, \epsilon))\right)^2 + 1\right],
\end{align}
where $c_6(n,\epsilon):=c_5(n, \epsilon)/c_4(n,\epsilon)$. Clearly, there exists an effectively computable positive constant $N_2(\epsilon) \geq N_1(\epsilon)$ 
such that (\ref{upper2}) holds for all $n \geq N_2(\epsilon)$.   

Let $N(\epsilon):=\max\{N, N_2(\epsilon)\}$. Then the inequalities (\ref{chenbound2}) hold for all $n \geq N(\epsilon)$. \qed

\section{Proof of Theorem \ref{main}}\label{s4}

\subsection{Proof of Conjecture (\ref{1})}  Let $\epsilon=1-\sqrt{6}/\pi$ in Theorem \ref{better}. We need to  
determine the constant $N(1-\sqrt{6}/\pi)$. A  
calculation shows that the inequality $c_4(n, 1-\sqrt{6}/\pi) > 0$ holds if $n \geq N_1(1-\sqrt{6}/\pi)$ where $N_1(1-\sqrt{6}/\pi)=4$. 
Next, we need to find the smallest positive integer $N_2(1-\sqrt{6}/\pi) \geq  4$ such that the bound 
\begin{align*}
n > \frac{1}{24}\left[\left(\frac{12}{\pi}\log(c_6(n, 1-\sqrt{6}/\pi))\right)^2 + 1\right]
\end{align*}
holds for all $n \geq N_2(1-\sqrt{6}/\pi)$. A calculation shows that this constant is given by $N_2(1-\sqrt{6}/\pi)=4845$.
We now have 
\begin{align*}
N(1-\sqrt{6}/\pi):=\max\{N, N_2(1-\sqrt{6}/\pi)\}=\max \{5310, 4845\}=5310.
\end{align*} 
Therefore, the inequalities 
\begin{align*}
\frac{\sqrt{6}}{\pi} \sqrt{n} p(n) < \mathrm{spt}(n) < \sqrt{n} p(n)
\end{align*}
hold for all $n \geq 5310$. Finally, one can verify with a computer that these inequalities also hold 
for $5 \leq n < 5310$. \qed

\subsection{Proof of Conjecture (\ref{2})} We follow closely the proof of \cite[Theorem 2.1]{BO}. By taking $\alpha=k=1$ in Theorem \ref{bounds} (recall 
that $B_1(1)=5729$), we find that
\begin{equation}\label{boundseq}
\frac{\sqrt{3}}{\pi\sqrt{24n-1}}\left(1-\frac{1}{n}\right)e^{\l(n)}< \spt(n) <\frac{\sqrt{3}}{\pi\sqrt{24n-1}}\left(1+\frac{1}{n}\right)e^{\l(n)}
\end{equation}
holds for all $n \geq 5729$.  One can verify with a computer that \eqref{boundseq} also holds for $1 \leq n < 5729$.  

Now, assume that $1 < a \leq b$, and let $b=Ca$ where $C \geq 1$. From (\ref{boundseq}) we get the inequalities
\begin{align*}
\spt(a)\spt(Ca) > \frac{3}{\pi^2 \sqrt{24a-1}\sqrt{24Ca-1}}\left(1-\frac{1}{a}\right)\left(1-\frac{1}{Ca}\right)e^{\l(a)+\l(Ca)}
\end{align*}
and
\begin{align*}
\spt(a+Ca) < \frac{\sqrt{3}}{\pi\sqrt{24(a+Ca)-1}}\left(1+\frac{1}{a+Ca}\right)e^{\l(a+Ca)}.
\end{align*}
Hence, for all but finitely many cases, it suffices to find conditions on $a>1$ and $C\geq1$ such that
\begin{equation}\label{Tinequality1}
e^{\l(a)+\l(Ca)-\l(a+Ca)} > 
\frac{ \pi \sqrt{24a-1}\sqrt{24Ca-1}}{\sqrt{3} \sqrt{24(a+Ca)-1}}\cdot\frac{\left(1+\frac{1}{a+Ca}\right)}{\left(1-\frac{1}{a}\right)\left(1-\frac{1}{Ca}\right)}.
\end{equation}

For convenience, define 
\begin{equation*}
T_a(C):=\l(a)+\l(Ca)-\l(a+Ca)\hspace{.5cm}\text{and}\hspace{.5cm}S_a(C):=\frac{\left(1+\frac{1}{a+Ca}\right)}{\left(1-\frac{1}{a}\right)\left(1-\frac{1}{Ca}\right)}.
\end{equation*}
Then by taking logarithms, we find that (\ref{Tinequality1}) is equivalent to 
\begin{equation}\label{T>S}
T_a(C)>\log\left(\frac{\pi \sqrt{24a-1}\sqrt{24Ca-1}}{\sqrt{3}\sqrt{24(a+Ca)-1}}\right)+\log(S_a(C)).
\end{equation}
As functions of $C$, it can be shown that $T_a(C)$ is increasing and $S_a(C)$ is decreasing for $C\geq1$, and thus
\begin{equation*}
T_a(C)\geq T_a(1)
\end{equation*}
and
\begin{equation*}
\log(S_a(1)) \geq \log (S_a(C)).
\end{equation*}
Hence it suffices to show that
\begin{equation*}
T_a(1) > \log\left(\frac{\pi \sqrt{24a-1}\sqrt{24Ca-1}}{\sqrt{3}\sqrt{24(a+Ca)-1}}\right)+\log(S_a(1)).
\end{equation*}
Moreover, since
\begin{equation*}
\frac{\sqrt{24Ca-1}}{\sqrt{24(a+Ca)-1}} \leq 1
\end{equation*}
for all $C\geq1$ and all $a>1$, it suffices to show that
\begin{equation}\label{T>S1}
T_a(1) > \log\left(\frac{\pi \sqrt{24a-1}}{\sqrt{3}}\right)+\log (S_a(1)).
\end{equation}
By computing the values $T_a(1)$ and $S_a(1)$, we find that \eqref{T>S1} holds for all $a \geq 6$.  
 
To complete the proof, assume that $2 \leq a\leq 5$.  For each such integer $a$, we 
calculate the real number $C_a$ for which
\begin{equation*}
T_a\left(C_a\right)=\log\left(\frac{ \pi \sqrt{24a-1}}{\sqrt{3}}\right)+\log (S_a\left(C_a\right)).
\end{equation*}
The values $C_a$ are listed in the table below.
\vspace{.2cm}
\begin{center}
\begin{tabular}{ |c|c| } 
 \hline
$a$ & $C_a$\\ 
 \hline
 \hline
 2 & 27.87\dots\\
 \hline
 3 & 3.54\dots\\
 \hline
 4 & 1.79\dots\\
\hline
5 & 1.20\dots\\
 \hline
\end{tabular}
\end{center}
\vspace{.2cm}
By the discussion above, if $b=Ca\geq a$ is an integer for which $C > C_a$, then (\ref{T>S}) holds, which in turn gives the theorem in these cases.  
Only finitely many cases remain, namely the pairs of integers where $2 \leq a \leq 5$ and $1 \leq b/a \leq C_a$.  We compute 
$\spt(a),\spt(b)$, and $\spt(a+b)$ in these cases to complete the proof.
\qed

\subsection{Proof of Conjecture (\ref{3})}

We require some lemmas and a proposition analogous to those of Desalvo and Pak \cite{PDS} in order to prove the remaining conjectures.  

The following is \cite[Lemma 2.1]{PDS}.
\begin{lemma}\label{pds1}
Suppose $h(x)$ is a positive, increasing function with two continuous derivatives for all $x>0$, and that $h'(x)>0$ is decreasing, 
and $h''(x)<0$ is increasing for all $x>0$.  Then for all $x > 0$, we have 
$$h''(x-1)<h(x+1)-2h(x)+h(x-1)<h''(x+1).$$
\end{lemma}

By Theorem \ref{riad}, we may write
\begin{equation*}
\spt(n)=f(n)+E_s(n)
\end{equation*}
where
\begin{equation}\label{sptsumdetails}
f(n):=\frac{\sqrt{3}}{\pi\sqrt{24n-1}}e^{\l(n)} 
\end{equation}
and 
\begin{align*}
|E_s(n)| < (3.59 \times 10^{22})2^{q(n)}(24n-1)^2e^{\l(n)/2}.
\end{align*}

\begin{lemma}\label{pds2}
Let 
\begin{align*}
F(n):=2\log (f(n))-\log (f(n+1))-\log (f(n-1)).
\end{align*} 
Then for all $n\geq4$, we have 
\begin{align*}
\frac{24\pi}{\big(24(n+1)-1\big)^{3/2}}-\frac{1}{n^2}<F(n)<\frac{24\pi}{\big(24(n-1)-1\big)^{3/2}} -\frac{288}{\big(24(n+1)-1\big)^2}.
\end{align*}
\end{lemma}

\begin{proof}
We can write $f(n)$ from (\ref{sptsumdetails}) as 
\begin{align*}
f(n)=\frac{1}{2\sqrt{3}\l(n)}e^{\l(n)},
\end{align*}
so that $$\log (f(n))=\l(n)-\log(\l(n))-\log(2\sqrt{3}).$$  Then we have 
\begin{equation*}
F(n)=2\l(n)-\l(n+1)-\l(n-1)-2\log(\l(n))+\log(\l(n+1))+\log(\l(n-1)).
\end{equation*}
Since the functions $\l(x)$ and $\widetilde{\l}(x):=\log(\l(x))$ satisfy the hypotheses of Lemma \ref{pds1}, we get
\begin{equation*}
-\l''(n+1)+\widetilde{\l}''(n-1) < F(n) < -\l''(n-1) + \widetilde{\l}''(n+1).
\end{equation*}
Computing derivatives gives 
\begin{align*}
\frac{24\pi}{\big(24(n+1)-1\big)^{3/2}}-\frac{288}{\big(24(n-1)-1\big)^2}<F(n)<\frac{24\pi}{\big(24(n-1)-1\big)^{3/2}}-\frac{288}{\big(24(n+1)-1\big)^2}
\end{align*}
for all $n \geq 2$, from which we deduce that
\begin{equation*}
\frac{24\pi}{\big(24(n+1)-1\big)^{3/2}}-\frac{1}{n^2}<F(n)<\frac{24\pi}{\big(24(n-1)-1\big)^{3/2}} - \frac{288}{\big(24(n+1)-1\big)^2} 
\end{equation*}
for all $n\geq4$.
\end{proof}

\begin{lemma}\label{pds3}
Define the functions $y_n:=\left|E_s(n)\right|/f(n)$, 
\begin{align*}
M(n):=2\sqrt{3}(3.59 \times 10^{22})\l(n)2^{q(n)}(24n-1)^2 e^{-\l(n)/2},
\end{align*} 
and 
\begin{align*}
g(n):=\frac{M(n)}{1-M(n)}.
\end{align*}
Then for all $n \geq 2$, we have  
\begin{align*}
\log\left[\frac{\left(1-y_n\right)^2}{\left(1+y_{n+1}\right)\left(1+y_{n-1}\right)}\right] &> -2g(n)-M(n+1)-M(n-1)
\end{align*}
and
\begin{align*}
\log\left[\frac{\left(1+y_n\right)^2}{\left(1-y_{n+1}\right)\left(1-y_{n-1}\right)}\right] &< 2M(n) + g(n+1) + g(n-1).
\end{align*}

\end{lemma}

\begin{proof} First observe that for all $n \geq 1$, we have 
\begin{align}\label{Mineq}
0 < y_n=\frac{\left|E_s(n)\right|}{f(n)} < \frac{(3.59 \times 10^{22})2^{q(n)}(24n-1)^2e^{\l(n)/2}}{\frac{1}{2\sqrt{3}\l(n)}e^{\l(n)}} = M(n).
\end{align}
The bound $M(n) < 1$ is equivalent to 
\begin{align}\label{Mineq2}
2\sqrt{3}(3.59 \times 10^{22})\l(n)2^{q(n)}(24n-1)^2 < e^{\l(n)/2} .
\end{align}
Clearly, there is an effectively computable positive integer $M_0$ such that the inequality (\ref{Mineq2}) holds for all $n \geq M_0$. 
A calculation shows that (\ref{Mineq2}) holds for all $n \geq M_0$ with $M_0=4698$. On the other hand, 
one can verify with a computer that  
\begin{align*}
\max_{1 \leq n < 4698}y_n < 1.
\end{align*}
Hence $y_n < 1$ for all $n \geq 1$. Then using (\ref{Mineq}) and the inequalities
\begin{align*}
\log(1-x)\geq-\frac{x}{1-x} \quad \textrm{for} \quad 0 < x < 1
\end{align*}
and 
\begin{align*}
\log(1+x) < x \quad \textrm{for} \quad x > 0, 
\end{align*}
we get 
\begin{align*}
\log\left[\frac{\left(1-y_n\right)^2}{\left(1+y_{n+1}\right)\left(1+y_{n-1}\right)}\right]
& =2\log\left(1-y_n\right)-\log\left(1+y_{n+1}\right)-\log\left(1+y_{n-1}\right)\\
& >-2 \frac{y_n}{1-y_n}-y_{n+1}-y_{n-1}\\
& > -2 g(n)  - M(n+1) - M(n-1)
\end{align*}
for all $n \geq 2$. Similarly, we get  
\begin{align*}
\log\left[\frac{\left(1+y_n\right)^2}{\left(1-y_{n+1}\right)\left(1-y_{n-1}\right)}\right] 
&= 2\log(1+y_n) - \log(1-y_{n+1}) - \log(1-y_{n-1})\\ 
& < 2 y_n  + \frac{y_{n+1}}{1-y_{n+1}} + \frac{y_{n-1}}{1-y_{n-1}}\\
& < 2M(n) + g(n+1) + g(n-1)
\end{align*}
for all $n \geq 2$.  
\end{proof}

\begin{proposition}\label{prop}
Let 
\begin{align*}
\spt_2(n):=2\log(\spt(n))-\log(\spt(n+1))-\log(\spt(n-1)).  
\end{align*}
Then we have 
\begin{align*}
\spt_2(n) > \frac{1}{(24n)^{3/2}}
\end{align*}
for all $n \geq 6553$ and 
\begin{align*}
\spt_2(n)<\frac{2}{n^{3/2}}
\end{align*}
for all $n \geq 6445$.
\end{proposition}

\begin{proof}
We first bound $\spt(n)$ by
\begin{equation*}
f(n)\left(1-\frac{\left|E_s(n)\right|}{f(n)}\right)<\spt(n)<f(n)\left(1+\frac{\left|E_s(n)\right|}{f(n)}\right).
\end{equation*}
Then recalling that 
\begin{align*}
F(n):=2\log (f(n))-\log (f(n+1))-\log (f(n-1))
\end{align*}
and 
$y_n:=\left|E_s(n)\right|/f(n)$, we take logarithms in the preceding inequalities to get 
\begin{align*}
F(n)+\log\left[\frac{\left(1-y_n\right)^2}{\left(1+y_{n+1}\right)\left(1+y_{n-1}\right)}\right]
< \spt_2(n) < F(n)+\log\left[\frac{\left(1+y_n\right)^2}{\left(1-y_{n+1}\right)\left(1-y_{n-1}\right)}\right].
\end{align*}
It follows immediately from Lemmas \ref{pds2} and \ref{pds3} that for all $n \geq 4$, we have
\begin{align*}
\spt_2(n) >  \frac{24\pi}{\big(24(n+1)-1\big)^{3/2}}-\frac{1}{n^2} -2g(n)-M(n+1)-M(n-1)
\end{align*}
and
\begin{align}\label{spt2upper}
\spt_2(n) < \frac{24\pi}{\big(24(n-1)-1\big)^{3/2}} -\frac{288}{\big(24(n+1)-1\big)^2} + 2 M(n) + g(n+1) +g(n-1).
\end{align}
Then a calculation shows that 
\begin{align*}
\frac{24\pi}{\big(24(n+1)-1\big)^{3/2}}-\frac{1}{n^2} -2g(n)-M(n+1)-M(n-1) > \frac{1}{(24n)^{3/2}}
\end{align*}
for all $n \geq 6553$ and 
\begin{align*}
\frac{24\pi}{\big(24(n-1)-1\big)^{3/2}} -\frac{288}{\big(24(n+1)-1\big)^2} +   2 M(n) + g(n+1) +g(n-1) < \frac{2}{n^{3/2}}
\end{align*}
for all $n \geq 6445$. This completes the proof.
\end{proof}

To prove Conjecture (\ref{3}), we must show that $$\spt(n)^2 > \spt(n-1)\spt(n+1)$$ 
for $n \geq 36$. Taking logarithms, 
we see that this is equivalent to $\spt_2(n) > 0.$  By the lower bound in Proposition \ref{prop}, we have $\spt_2(n)>0$ 
for all $n \geq 6553$. Finally, one can verify with a computer that 
$\spt_2(n)>0$ for all $36 \leq n < 6553$. This completes the proof.
\qed

\subsection{Proof of Conjecture (\ref{4})} We follow closely the proof of \cite[Theorem 5.1]{PDS}. 
Recall that a sequence $\{a(k)\}_{k=0}^{\infty}$ of non-negative integers is log-concave if 
\begin{align*}
a(k)^2 \geq a(k-1)a(k+1)
\end{align*}
for all $k \geq 1$.
Moreover, it is known that log-concavity implies strong log-concavity 
\begin{align*}
a(\ell-i)a(k+i) \geq a(k)a(\ell),
\end{align*}
for all $0\leq k \leq \ell \leq n$ and $0 \leq i \leq \ell - k$ (see e.g. \cite{S}). 

Now, we have proved that 
\begin{align*}
\spt(n)^2 > \spt(n-1)\spt(n+1)
\end{align*}
for all $n \geq 36$. Therefore, if we take $k=n-m$, $\ell=n+m$, and $i=m$, then 
\begin{equation*}
\spt(n)^2 > \spt(n-m)\spt(n+m)
\end{equation*}
for all  $n>m>1$ with $n-m > 36$. 

We next consider the case $n>m>1$ with $1 \leq n-m \leq 36$. We will prove that 
\begin{align}\label{string}
\spt(n)^2 \geq \spt(m+1)^2  > \spt(36)\spt(36+2m) \geq \spt(n-m)\spt(n+m)
\end{align}
for all $1 \leq n-m \leq 36$ with $m \geq 6244$. On the other hand, one can verify with a computer that 
\begin{align*}
\spt(n)^2 > \spt(n-m)\spt(n+m) 
\end{align*} 
for all $1 \leq n-m \leq 36$ with $m < 6244$. This completes the proof of Conjecture (\ref{4}), subject to verifying the inequalities (\ref{string}).

Since $n \geq m+1$, we have
\begin{align*}
\spt(n)^2 \geq \spt(m+1)^2.
\end{align*}
Moreover, since $n-m \leq 36$ we have 
\begin{align*}
\spt(n-m)< \spt(36),
\end{align*}
and thus 
\begin{align*}
\spt(36)\spt(36+2m) \geq \spt(n-m)\spt(n+m).
\end{align*}
This verifies the first and third inequalities in (\ref{string}). 

It remains to prove that 
\begin{align}\label{sptlast}
\spt(m+1)^2 > \spt(36)\spt(36+2m)
\end{align}
for all $m \geq 6244$.  Taking logarithms in (\ref{sptlast}), we see that it suffices to prove 
\begin{align}\label{sptlast2}
2\log (\spt(m+1)) - \log (\spt(36)) - \log (\spt(36+2m)) > 0
\end{align}
for all $m \geq 6244$. By \cite[Section 2]{HR} and \cite[(4)]{E}, respectively, we have the lower and upper bounds 
\begin{align*}
\frac{e^{2\sqrt{m}}}{2\pi m e^{1/6m}} < p(m) < e^{\pi \sqrt{(2/3)m}}
\end{align*}
for all $m \geq 1$. Then by the inequality stated in Conjecture (\ref{1}) (which is true by Theorem \ref{main}), we have
\begin{equation}\label{sptsqueeze}
\frac{\sqrt{6}}{\pi} \sqrt{m} \frac{e^{2\sqrt{m}}}{2\pi me^{1/6m}} < \spt(m) < \sqrt{m}e^{\pi \sqrt{(2/3)m}}
\end{equation}
for all $m \geq 5$. Using the inequalities (\ref{sptsqueeze}) and $\spt(36)<90000$, we see that the left hand side of (\ref{sptlast2}) is 
bounded below by the function
\begin{equation*}
2\log\left(\frac{\sqrt{6(m+1)}}{2\pi^2(m+1)e^{1/6(m+1)}}\right)+4\sqrt{m+1}-\log(90000)-\log(\sqrt{36+2m})-\frac{\pi\sqrt{2(36+2m)}}{\sqrt{3}}
\end{equation*}
for all $m \geq 4$. A calculation shows that this function is positive for all $m \geq 6244$.  \qed

\subsection{Proof of Conjecture (\ref{5})}

Taking logarithms, we find that Conjecture (\ref{5}) is equivalent to  
$$\spt_2(n)<\log\left(1+\frac{1}{n}\right)$$  
for all $n \geq 13$. By the upper bound in Proposition \ref{prop} and some straightforward estimates, 
we have $$\spt_2(n)<\frac{2}{n^{3/2}}<\frac{1}{n+1}<\log\left(1+\frac{1}{n}\right)$$ for all $n \geq 6445$.  Finally, 
one can verify with a computer that the conjectured inequality 
holds for all $13 \leq n < 6445$. This completes the proof.
\qed

\subsection{Proof of Conjecture (\ref{6})} We follow closely the proof of \cite[Conjecture 1.3]{CWX}. 
Taking logarithms, we find that Conjecture (\ref{6}) is equivalent to
\begin{equation*}
\spt_2(n)<\log\left(1+\frac{\pi}{\sqrt{24}n^{3/2}}\right)
\end{equation*}
for all $n \geq 73$. By (\ref{spt2upper}) we have 
\begin{align*}
\spt_2(n) < \frac{24\pi}{\big(24(n-1)-1\big)^{3/2}} -\frac{288}{\big(24(n+1)-1\big)^2} + 2 M(n) + g(n+1) +g(n-1)
\end{align*}
for all $n \geq 4$. On the other hand, by \cite[(2.3)]{CWX} we have
\begin{equation*}
\frac{24\pi}{\big(24(n+1)-1\big)^{3/2}}<\frac{24\pi}{(24n)^{3/2}}-\left(\frac{24\pi}{(24n)^{3/2}}\right)^2+\frac{3}{2n^{5/2}}
\end{equation*}
for all $n \geq 50$, and by \cite[(2.23)]{CWX} we have
\begin{equation*}
-\frac{288}{\big(24(n+1)-1\big)^2}<\frac{1}{6n^{5/2}}-\frac{1}{2n^2}
\end{equation*}
for all $n \geq 50$.  Therefore, for all $n \geq 50$ we have 
\begin{equation*}
\spt_2(n)<\frac{24\pi}{(24n)^{3/2}}-\left(\frac{24\pi}{(24n)^{3/2}}\right)^2 + \frac{5}{3n^{5/2}} - \frac{1}{2n^2} + 2 M(n) + g(n+1) +g(n-1).
\end{equation*}
Now, a calculation shows that 
\begin{align*}
\frac{5}{3n^{5/2}} - \frac{1}{2n^2} + 2 M(n) + g(n+1) +g(n-1) < 0
\end{align*}
for all $n \geq 7211$. Hence
\begin{align*}
\spt_2(n) < \frac{24\pi}{(24n)^{3/2}}-\left(\frac{24\pi}{(24n)^{3/2}}\right)^2 = \frac{24\pi}{(24n)^{3/2}}\left(1-\frac{24\pi}{(24n)^{3/2}}\right)
\end{align*}
for all $n \geq 7211$. Then using the inequality 
\begin{align*}
x(1-x)<\log(1+x) \quad \textrm{for} \quad x>0,
\end{align*}
we get 
\begin{align*}
\spt_2(n) < \log\left(1+ \frac{24\pi}{(24n)^{3/2}}\right)=\log\left(1+\frac{\pi}{\sqrt{24}n^{3/2}}\right)
\end{align*}
for all $n \geq 7211$. Finally, one can verify with a computer that this inequality also holds for all $73 \leq n < 7211$. 
This completes the proof. \qed

\end{document}